\documentclass[a4paper,11pt]{article}

\usepackage{color}
\usepackage[latin1]{inputenc}
\usepackage[T1]{fontenc}
\usepackage[normalem]{ulem}
\usepackage[english]{babel}
\usepackage{verbatim}
\usepackage{graphicx}
\usepackage{enumerate,enumitem}
\usepackage{amsmath,amssymb,amsfonts,amsthm,mathrsfs}
\usepackage{rotating,relsize}
\usepackage{float}
\usepackage{array}
\usepackage{MnSymbol}
\usepackage[all,cmtip]{xy}
\usepackage[hidelinks]{hyperref}
\xyoption{all}
\setlist[enumerate]{leftmargin=.7cm,label=\roman*)}
\usepackage{mathbbol}
\usepackage{authblk}
\usepackage{enumerate}

\newtheorem{theorem}{Theorem}[section]
\newtheorem{theoremA}{Theorem}

\newtheorem*{theorem*}{Theorem}
\newtheorem*{cor*}{Corollary}

\newtheorem{lemma}[theorem]{Lemma}
\newtheorem{prop}[theorem]{Proposition}
\newtheorem{cor}[theorem]{Corollary}

\theoremstyle{definition}
\newtheorem*{remark*}{Remark}
\newtheorem*{rem*}{Remark}
\newtheorem{defn}[theorem]{Definition}
\newtheorem{rem}[theorem]{Remark}
\newtheorem{example}[theorem]{Example}

\newcommand{\xra}{\xrightarrow}

\DeclareMathOperator{\gl}{gl}

\DeclareMathOperator{\Hom}{Hom}

\DeclareMathOperator{\Z}{\mathbb{Z}}

\DeclareMathOperator{\orb}{orb}

\DeclareMathOperator{\Gen}{Gen}

\newcommand{\bs}{\backslash}

\newcommand{\un}{\underline}
\newcommand{\td}[1]{\langle #1\rangle}

\title{}
\author[1]{Irakli Patchkoria}
\affil[1]{}
\date{}

\begin{document}

\begin{center}\LARGE{Chromatic congruences and Bernoulli numbers}
\end{center}

\begin{center}\Large{Irakli Patchkoria}
\end{center}

\begin{center}\emph{\footnotesize{Dedicated to the memory of my great granduncle Terenti Shamugia,\\a talented mathematician who passed away very young}}
\end{center}

\vspace{.05cm}

\abstract{For every natural number $n$ and a fixed prime $p$, we prove a new congruence for the orbifold Euler characteristic of a group. The $p$-adic limit of these congruences as $n$ tends to infinity recovers the Brown-Quillen congruence. We apply these results to mapping class groups and using the Harer-Zagier formula we obtain a family of congruences for Bernoulli numbers. We show that these congruences in particular recover classical congruences for Bernoulli numbers due to Kummer, Voronoi, Carlitz, and Cohen.} 

\vspace{.05cm}



\section{Introduction} The \emph{orbifold Euler characteristic} $\chi_{\orb}(G)$ of a group is a well-studied invariant which has applications in topology, group theory, and number theory. It was first defined by Wall \cite{Walleuler}. Harder \cite{Hard} and Serre \cite{Serreeuler} related it to number theory by computing the orbifold Euler characteristic of arithmetic groups in terms of Dedekind zeta functions. 

Given a discrete group $G$, the invariant $\chi_{\orb}(G)$ is a rational number defined under certain conditions on $G$. The conditions we will use in this paper are the following: $G$ is virtually torsion-free and admits a finite classifying $G$-space $\un{E}G$ for proper actions. The latter is uniquely determined up to a $G$-equivariant homotopy equivalence with the properties: 
\begin{itemize}

\item $\underline{E}G$ is a $G$-CW-complex; 

\item The $H$-fixed subspace $\underline{E}G^H$ is contractible if $H \leq G$ is finite and empty otherwise. 

\end{itemize}

For any discrete group $G$, a model of $\un{E}G$ always exists. We say that $G$ \emph{admits a finite $\un{E}G$} if there exists a finite $G$-CW complex model of $\un{E}G$. The orbifold Euler characteristic can be defined for more general groups but the vast majority of examples one cares about in geometric group theory and number theory satisfy the above conditions. For more details and examples see e.g., \cite{Lsurvey}.

Under the additional assumption that any elementary abelian $p$-subgroup of $G$ has the rank at most $1$, Brown in \cite{KBro1}  and \cite{KBro3} proved a congruence 
\[\chi_{\orb}(G) \equiv \frac{1}{p-1} \sum_{[g]} \chi_{\orb}(C\td{g}) \mod \Z_{(p)},\]
where $[g]$ runs over the conjugacy classes of order $p$ elements and $C\td{g}$ denotes the centralizer of $g$. Brown applied this formula to certain symplectic groups over number rings and using theorems of Harder and Serre, provided new results about the denominators of special values of Dedekind zeta functions. An interesting special case is given by the symplectic group $G=Sp_{p-1}(\mathbb{Z})$ where $p$ is an odd prime. In this case using \cite{Hard}, Brown's congruence recovers the congruence
\[\zeta(-1) \cdot \zeta(-3) \cdot \cdots \cdot \zeta(2-p) \equiv \frac{2^{\frac{p-3}{2}}}{p(p-1)} \cdot h_p^{-} \mod \Z_{(p)},\]
where $\zeta$ is the Riemann zeta function and $h_p^{-}$ is the first factor of the class number of $\mathbb{Q}(\zeta_p)$. Using the formula $\zeta(1-2n)=-\frac{B_{2n}}{2n}$, where $B_{2n}$ is $2n$-th Bernoulli number and the von Staudt-Clausen theorem \cite{Staudt, Clausen}, one concludes that $p$ does not divide the first factor $h_p^{-}$ if an only if $p$ does not divide the numerator of any of the Bernoulli numbers $B_2, B_4, \dots, B_{p-3}$. This is a well-known criterion for regularity proved by Kummer in \cite{Kumregular}. 

Later Quillen generalized Brown's congruence and removed the restriction on the rank of elementary abelian $p$-subgroups. Quillen showed that if $G$ is a virtually torsion-free group with a finite $\un{E}G$, then one has
\[ \sum_{(E)} (-1)^{r(E)}p^{\binom{r(E)}{2}}\chi_{\orb}(N(E)) \equiv 0 \mod \Z_{(p)},\]
where the sum runs over the conjugacy classes of elementary abelian $p$-subgroups, $r(E)$ denotes the rank of $E$, and $N(E)$ the normalizer of $E$. The proof was written up by Brown in \cite[Theorem 14.2]{Brownbook}. We refer to this congruence as the \emph{Brown-Quillen congruence}. The proof of Quillen and Brown relies on the homotopical and homological analysis of the geometric realization of the poset of non-trivial elementary abelian $p$-subgroups. In particular, it uses a version of the Solomon-Tits theorem and spectral sequence arguments. 

The goal of this paper is to provide new congruences for $\chi_{\orb}(G)$ and apply them to the mapping class group $\Gamma_u$ of the closed oriented surface of genus $u$. Following \cite{HKR} we denote by $G_{n,p}$ the set of $n$-tuples $(g_1, \dots, g_n)$ of commuting elements ($g_i g_j=g_jg_i$ for all $i$ and $j$), each of which has $p$-power order. The group $G$ acts on $G_{n,p}$ by conjugation in each coordinate. If $G$ admits a finite $\un{E}G$, then the quotient set $G \bs G_{n,p}$ is finite. The following is the first main result of this paper:

\begin{theoremA} \label{theorema} Let $p$ be a prime and $n \geq 1$. Suppose $G$ is a virtually torsion-free discrete group with a finite $\underline{E}G$. Then we have
\[\sum_{[g_1,\dots,g_n] \in G \bs G_{n,p}} \chi_{\orb}(C\td{g_1, \dots, g_n}) \equiv 0 \mod \Z_{(p)},\]
where $C\td{g_1, \dots, g_n}$ denotes the centralizer of the subgroup generated by the tuple $(g_1, \dots, g_n)$. 
\end{theoremA}
The proof of this theorem uses a classical fact from group theory known as the Frobenius theorem \cite{Frobenius}: If $n$ divides the order of a finite group $H$, then the number of solutions of the equation $x^n=1$ in $H$ is a multiple of $n$. 

\begin{remark*} We refer to the congruence in Theorem \ref{theorema} as \emph{the chromatic congruence at height $n$}, because the orbit set $H \bs H_{n,p}$ plays a prominent role in the chromatic homotopy theory. By the work of Hopkins-Kuhn-Ravenel, for $H$ a finite group and $n \geq 1$, the cardinality $\vert H \bs H_{n,p} \vert$ is equal to the Morava $K$-theory Euler characteristic $\chi_{K(n)}(BH)$ of the classifying space $BH$ \cite[Theorem B]{HKR}. The $n$-th Morava $K$-theory spectrum $K(n)$ is a complex oriented theory with a formal group law of height $n$. 
\end{remark*} 

To relate the congruence of Theorem \ref{theorema} to the Brown-Quillen congruence we use a result of Hall from \cite{Hall} which uses the M\"obius inversion for posets. In particular, Hall's result computes the cardinality of the set of generating $n$-tuples of a finite abelian $p$-group. An account can be found in \cite[Section 6.B]{Dia} where this set shows up as the state space of a natural Markov chain for random generation. Reformulating Theorem \ref{theorema} in terms of the generating tuples yields:

\begin{theoremA} \label{theoremb}  Let $p$ be a prime and $n \geq 1$. Suppose $G$ is a virtually torsion-free discrete group with a finite $\underline{E}G$. Then we have
\[\sum_{(H)} \chi_{\orb} (N(H)) \cdot  \vert \Phi(H) \vert^n \cdot \big ( \sum_{i=0}^{r(H/\Phi(H))} (-1)^i p^{n(r(H/\Phi(H))-i)+\binom{i}{2}} \cdot \binom{r(H/\Phi(H))}{i}_p  \big) \equiv 0 \mod \Z_{(p)},\]
where the sum runs over the conjugacy classes of finite abelian $p$-subgroups generated by at most $n$ elements. Here $N(H)$ denotes the normalizer of $H$, $\Phi(H)$ the Frattini subgroup of $H$, $\binom{a}{b}_p$ is the $p$-binomial coefficient, and $r$ is the rank. \end{theoremA}

This theorem in particular provides a new and in some sense an elementary proof of the Brown-Quillen congruence. Below we show that the $p$-adic limit of the left hand side exists as $n \to \infty$ and the limit is equal to the left hand side of the Brown-Quillen congruence. Our proof only uses the abstract existence of a finite $\un{E}G$ and combinatorial arguments. In particular, we fully avoid to deal with the homotopy or homology type of the poset of non-trivial elementary abelian $p$-subgroups. 

\subsection*{Application} We apply Theorem \ref{theorema} to the mapping class group $\Gamma_u$ of a closed oriented surface $S_u$ of genus $u$. By \cite{HZ}, one has 
\[\chi_{\orb}(\Gamma_u)=\frac{\zeta(1-2u)}{2-2u}=\frac{B_{2u}}{2u(2u-2)}.\]
Hence using Theorem \ref{theorema}, one obtains an infinite family of congruences for Bernoulli numbers. We make these congruences explicit at height $n=1$. Turns out that they recover the following classical congruences for a prime $p \geq 5$:

\begin{itemize} 

\item The congruence 
\[\frac{B_{2u}}{2u} \equiv \frac{B_{2v}}{2v} \mod p^r\Z_{(p)},\]
for $2u=p^rx+2$ and $2v=p^{r-1}x+2$, where $r, x \geq 1$ and $p-1 \nmid 2u$. This is a special case of Kummer's congruence \cite{Kum, Vor}, proved in the original form by Kummer and then generalized by Voronoi. 

\item The congruence 
\[B_{2u} +\frac{1}{p} \equiv 1 \mod p^r\Z_{(p)},\]
for $2u=xp^r(p-1)$, where $r\geq 0$, $x \geq 1$, and $p \; \nmid \; 2u-2$. This is a congruence due to Carlitz \cite{Carl} generalizing the $p$-local von Staudt-Clausen theorem \cite{Staudt, Clausen}. 

\item The congruence
\[\frac{B_{2u}}{2u}- \frac{B_{2v}}{2v} \equiv \Big (\frac{1}{2u}-\frac{1}{2v} \Big) \Big(1-{\frac{1}{p}}\Big) \mod p^r\Z_{(p)},\]
for $u=p^r\Big( \frac{p-1}{2}k-1  \Big)+1$ and $v=p^{r-1}\Big( \frac{p-1}{2}k-1  \Big)+1$, where $r, k \geq 1$. This congruence seems to be less known than the previous two. To our knowledge the only written down proof in the literature is due to Cohen  \cite[Proposition 11.4.4]{Cohen}. Cohen's proof crucially uses the theory of $p$-adic $L$-functions. Our proof is very different. We instead use the geometry of the moduli space (e.g., the results of \cite{HZ}) and Theorem \ref{theorema}. 
\end{itemize}

To apply Theorem \ref{theorema}, we require the knowledge of abelian $p$-subgroups. We can freely choose the number $n$ and this allows us to have some control on the cardinalities of generating sets. For example, the congruence at height $n=1$ requires only the knowledge of cyclic $p$-subgroups. On the other hand, to use the Brown-Quillen congruence, we require the knowledge of all elementary abelian $p$-subgroups. There are cases when it is easier to understand cyclic $p$-subgroups than elementary abelian $p$-subgroups. This seems to be the case for mapping class groups. By the work of Harer-Zagier \cite{HZ}, we have all the data available to apply Theorem  \ref{theorema} at height $n=1$. On the other hand, we do not have a complete classification of elementary abelian $p$-subgroups and their centralizers. However, the papers \cite{Broughton2, Broughton3} have made significant progress in this direction and one could try to use them to make the Brown-Quillen congruence explicit for all mapping class groups. This has not yet been addressed in the literature. 

A good example to illustrate the difference between various congruences in this paper is the mapping class group $\Gamma_{\frac{(p-1)(p-2)}{2}}$ for a prime $p \geq 5$. For any $n \geq 1$, Theorem \ref{theoremb} gives a congruence
\[\frac{\zeta(1-2u)}{2-2u}+\frac{p^n-1}{p-1}\Big(-\frac{1}{p}\cdot \frac{\zeta(1-2v)}{2-2v} + \frac{(p-1)^2}{2u \cdot p(2u+2p-2)} \Big )+(p^{2n}-(1+p)p^n+p)\cdot \frac{1}{6p^2} \equiv 0 \mod \Z_{(p)},\]
where $u=\frac{(p-1)(p-2)}{2}$ and $v=\frac{p-1}{2}$. If we take the $p$-adic limit as $n \to \infty$, we get the Brown-Quillen congruence
\[\frac{\zeta(1-2u)}{2-2u}-\frac{1}{p-1}\Big(-\frac{1}{p}\cdot \frac{\zeta(1-2v)}{2-2v} + \frac{(p-1)^2}{2u \cdot p(2u+2p-2)} \Big )+\frac{1}{6p} \equiv 0 \mod \Z_{(p)}.\]
In the special case $n=1$ one gets:
\[ \frac{\zeta(1-2u)}{2-2u}-\frac{1}{p}\cdot \frac{\zeta(1-2v)}{2-2v} + \frac{(p-1)^2}{2u \cdot p(2u+2p-2)} \equiv 0 \mod \Z_{(p)}.\]
The latter does not require the knowledge of elementary abelian $p$-subgroups of rank bigger than $1$. However, the formulas for $n \geq 2$ and the Brown-Quillen congruence need the higher rank subgroups. Here the number theoretic difference between $n=1$ and $n \geq 2$ cases can be understood through the von Staudt-Clausen theorem, as discussed at the end of the paper.

We expect that the general congruences provided in this paper will have further applications in number theory. These could especially come from arithmetic groups where cyclic subgroups are easier to understand than the elementary abelian subgroups. Even more specifically, we are interested in arithmetic groups whose nontrivial cyclic $p$-subgroups are all isomorphic to $\Z/p$ but additionally also have an elementary abelian $p$-subgroup of rank higher than $1$. An example of such a group is the symplectic group $Sp_{2(p-1)}(\Z)$. 

\subsection*{Acknowledgments} I would like to thank Wolfgang L\"uck, Oscar Randal-Williams, and Stefan Schwede for helpful conversations. This research was supported by the EPSRC grant EP/X038424/1 ``Classifying spaces, proper actions and stable homotopy theory''.

\section{Preliminaries}

A \emph{proper} $G$-CW complex is a $G$-CW complex $X$ such that the stabilizers of cells are finite. In particular, the $G$-space $\un{E}G$ is a proper $G$-CW complex. A proper $G$-CW complex $X$ is called \emph{finite} if it has finitely many equivariant cells. 

In what follows $N(H)$ denotes the normalizer of $H \leq G$ and $C(H)$ denotes the centralizer of $H$. We will denote by $W(H)$ the quotient $W(H)=N(H)/C(H)$ often referred to as the \emph{Weyl group} of $H$. If $H$ is finite, then  $W(H)$ is finite. 

\begin{prop} \label{lemma: finite model for centralizers}
  Let $X$ be a finite proper $G$-CW complex
  and $H \leq G$ a finite subgroup.
  Then the $C(H)$-space $X^{H}$ is a finite proper $C(H)$-CW complex. 
\end{prop}
\begin{proof} It is enough to check that for any finite subgroup $K\leq G$,
  the $H$-fixed point set $(G/K)^H$ as a $C(H)$-set has only finite stabilizers
  and finitely many orbits. To show this we consider a map
  \[ C(H) \bs (G/K)^H \to \Hom(H,K)/K ,  \]
  where $\Hom(H,K)$ is the set of group homomorphisms and $K$ acts by the conjugation. This map sends $[x K]$ to $c_x : H \to K$, $c_x(h)=x^{-1}hx$. It is well defined and injective and since the target is finite, so is the source. The stabilizer of $xK \in (G/K)^H$ is given by $C(H) \cap xKx^{-1}$ which is also finite. \end{proof}
  
\begin{cor} \label{Normalizer and Centralizer} Suppose $G$ admits a finite $\un{E}G$ and let $H \leq G$ be a finite subgroup. Then $\un{E}G^H$ is a finite model of $\un{E}N(H)$ and  $\un{E}C(H)$. 
\end{cor}

The following definition is due to Wall \cite{Walleuler}:

\begin{defn} \label{orbidef} Let $G$ be a virtually torsion-free discrete group with $\Lambda \leq G$ a finite index torsion-free subgroup and $X$ a finite proper $G$-CW complex. Then the \emph{orbifold Euler characteristic} of $X$, denoted by $\chi^G_{\orb}(X)$, is defined to be the rational number
\[\frac{\chi_\mathbb{Q}(\Lambda \bs X)}{[G:\Lambda]},\]
where $\chi_\mathbb{Q}$ is the classical Euler characteristic. 
If $G$ admits a finite $\un{E}G$, then $\chi^G_{\orb}(\un{E}G)$ is called the \emph{orbifold Euler characteristic of $G$} and is denoted by $\chi_{\orb}(G)$. 
\end{defn}
We observe that 
\[\chi_{\orb}(G)=\frac{\chi_\mathbb{Q}(B\Lambda)}{[G:\Lambda]},\]
since $\un{E}G$ as a $\Lambda$-space is a finite model of $E \Lambda$, the free contractible $\Lambda$-CW complex. Here $B\Lambda=\Lambda \bs E \Lambda$ is the classifying space. This definition is independent of the choice of $\Lambda$. 

The following is known as Quillen's formula (see e.g., \cite[Proposition 7.3]{Brownbook}):

\begin{prop} \label{Quillenformula}
Let $G$ be a virtually torsion-free discrete group.
  \begin{enumerate}
  \item
    For every finite proper $G$-CW complex $X$, one has
    \[\chi^G_{\orb}(X)=\sum_{G\sigma} (-1)^{n_\sigma} \frac{1}{\vert H^{\sigma} \vert},\]
    where the sum runs over all $G$-orbits of cells $\sigma$ of $X$, the number $n_\sigma$ is the dimension of $\sigma$, and $H^{\sigma}$ is the stabilizer of $\sigma$. 
  \item
    If $G$ admits a finite model for $\underline{E}G$, then one has
    \[\chi_{\orb}(G)=\sum_{G\sigma} (-1)^{n_\sigma} \frac{1}{\vert H^{\sigma} \vert},\] 
    where the sum runs over all $G$-orbits of cells $\sigma$ of a fixed finite $G$-CW model for $\underline{E}G$, the number $n_\sigma$ is the dimension of $\sigma$, and $H^{\sigma}$ is the stabilizer of $\sigma$.
  \end{enumerate}
\end{prop}

\section{Chromatic congruences and the proof of Theorem \ref{theorema}} 

In this section we prove the main general result of this paper. 

\begin{prop} \label{chromaticidentity} Let $G$ be a virtually torsion-free discrete group with a finite $\underline{E}G$. Then for any finite proper $G$-CW complex $X$ and $n \geq 1$, we have
\[\sum_{[g_1,\dots,g_n] \in G \bs G_{n,p}} \chi_{\orb}^{C\td{g_1, \dots, g_n}}(X^{\td{g_1, \dots, g_n}})=\sum_{G\sigma} (-1)^{n_\sigma} \frac{\vert H^{\sigma}_{n,p}\vert} {\vert H^{\sigma} \vert},\]
where the sum runs over all $G$-orbits of cells $\sigma$ of $X$, the number $n_\sigma$ is the dimension of $\sigma$, and $H^{\sigma}$ is the stabilizer of $\sigma$.

\end{prop}

\begin{proof} Let $\Lambda$ be a finite index torsion-free subgroup of $G$. We will use the following observation on $G$-sets: Let $Y$ be a $G$-set such that $G \bs Y$ is finite. Then 
\begin{equation} \label{cosetequation}\vert  \Lambda \bs Y \vert = \sum_{[y] \in G \bs Y} \vert \Lambda \backslash G/G_y \vert, \end{equation}
where $G_y$ is the stabilizer of $y$. We also use that for any finite subgroup $H \leq G$, the map 
\begin{equation} \label{tuplesbijection}\coprod_{(g_1,\dots,g_n) \in G_{n,p}^{rep}} C\td{g_1, \dots, g_n} \bs (G/H)^{\td{g_1, \dots, g_n}} \to H \bs H_{n,p},\end{equation}
sending $[xH] \in C\td{g_1, \dots, g_n} \bs (G/H)^{\td{g_1, \dots, g_n}}$ to $[x^{-1}g_1x, \dots x^{-1}g_nx]$ is a bijection, where $G_{n,p}^{rep}$ is a fixed set of representatives of $G \bs G_{n,p}$. 

Both sides of the desired identity are additive functions in $X$. It suffices to check the formula for $X=G/H$, where $H$ is finite. We have (below $C=C\td{g_1, \dots, g_n}$ for short and $C_H$ denotes the centralizer in $H$)
\begin{align*}& \sum_{[g_1,\dots,g_n] \in G \bs G_{n,p}} \chi_{\orb}^{C\td{g_1, \dots, g_n})}((G/H)^{\td{g_1, \dots, g_n})})\\&= \sum_{[g_1,\dots,g_n] \in G \bs G_{n,p}} \frac{1}{[C\td{g_1, \dots, g_n}: (C\td{g_1, \dots, g_n} \cap \Lambda)]} \vert (C\td{g_1, \dots, g_n} \cap \Lambda) \bs (G/H)^{\td{g_1, \dots, g_n}} \vert\\&=\sum_{[g_1,\dots,g_n] \in G \bs G_{n,p}} \frac{1}{[C: C \cap \Lambda]} \;\;\sum_{[x] \in C \bs (G/H)^{\td{g_1, \dots, g_n}} } \vert (C \cap \Lambda)\backslash C/(C\cap xHx^{-1}) \vert \\&= \sum_{[g_1,\dots,g_n] \in G \bs G_{n,p}} \;\;\sum_{[x] \in C \bs (G/H)^{\td{g_1, \dots, g_n}}} \frac{1}{\vert C\cap xHx^{-1} \vert} \\&= \sum_{[h_1,\dots,h_n] \in H \bs H_{n,p}} \frac{1}{\vert C_H\td{h_1, \dots, h_n}\vert} =  \frac{\vert H_{n,p}\vert} {\vert H \vert}.\hspace{3cm}  \end{align*}
Here we used Equation (\ref{cosetequation}) for the second identity, the bijection (\ref{tuplesbijection}) for the fourth identity and the class equation for the last identity. \end{proof}

Recall that a rational number $a$ is called \emph{$p$-integral} if $a$ belongs to $\Z_{(p)}$, or equivalently if $a=0 \in \mathbb{Q}/\Z_{(p)}$, i.e., the congruence $a \equiv 0 \mod \Z_{(p)}$ holds. 

\begin{lemma}\label{Frobenius} Let $H$ be a finite group. Then the rational number
\[\frac{\vert H_{n,p}\vert} {\vert H \vert}\]
is $p$-integral.
\end{lemma}

\begin{proof} Let $r$ be the maximal number such that $p^r$ divides $\vert H \vert$. Then $\vert H \vert=p^rk$, where $\gcd(k,p)=1$ and 
\[H_{1,p}=\{x \in H \;\; \vert \;\; x^{p^r}=1\}.\]
By the Frobenius theorem \cite{Frobenius}, the cardinality $\vert H_{1,p} \vert$ is divisible by $p^r$. We have
\[\frac{\vert H_{1,p}\vert} {\vert H \vert}=\frac{\vert H_{1,p}\vert} {p^r} \cdot \frac{1} {k},\]
where $\vert H_{1,p}\vert / {p^r}$ is an integer and $1/k$ is $p$-integral. This completes the proof for the case $n=1$. 

Now let $n \geq 2$. We have the following chain of identities:
\begin{align*}&\frac{\vert H_{n,p}\vert} {\vert H \vert}=\sum_{(h_1, \dots, h_{n-1}) \in H_{n-1,p} } \frac{\vert {C\td{h_1, \dots, h_{n-1}}}_{1,p}\vert} {\vert H \vert}  \\ &= \sum_{(h_1, \dots, h_{n-1}) \in H_{n-1,p} } \frac{\vert {C\td{h_1, \dots, h_{n-1}}}_{1,p}\vert} {\vert C\td{h_1, \dots, h_{n-1}} \vert} \cdot \frac{\vert C\td{h_1, \dots, h_{n-1}} \vert}{\vert H \vert} \\&=
\sum_{[h_1, \dots, h_{n-1}] \in H \bs H_{n-1,p}} \frac{\vert {C\td{h_1, \dots, h_{n-1}}}_{1,p}\vert} {\vert C\td{h_1, \dots, h_{n-1}} \vert}. \end{align*}
By the previous paragraph the terms in the last sum are $p$-integral which completes the proof. \end{proof}

\begin{example} In general the numerator of $\frac{\vert H_{n,p}\vert} {\vert H \vert}$ might not be divisible by $p$. For example, let $H=\Sigma_3$, $p=2$, and $n=2$. Then
$\frac{\vert H_{2,2}\vert} {\vert H \vert}=\frac{5}{3}$. 
\end{example}

We are now ready to prove Theorem \ref{theorema} about the chromatic congruence: 

\begin{theorem} \label{maintheorem} Let $G$ be a virtually torsion-free discrete group with a finite $\underline{E}G$. Then for any finite proper $G$-CW complex $X$ and $n \geq 1$, we have
\[\sum_{[g_1,\dots,g_n] \in G \bs G_{n,p}} \chi_{\orb}^{C\td{g_1, \dots, g_n}}(X^{\td{g_1, \dots, g_n}}) \equiv 0 \mod \Z_{(p)}.\]
In particular, if we take $X=\underline{E}G$, we get
\[\sum_{[g_1,\dots,g_n] \in G \bs G_{n,p}} \chi_{\orb}(C\td{g_1, \dots, g_n}) \equiv 0 \mod \Z_{(p)}.\]
\end{theorem}

\begin{proof} This follows by combining Proposition \ref{chromaticidentity} and Lemma \ref{Frobenius}. The second claim follows since if $X=\underline{E}G$, then $X^{\td{g_1, \dots, g_n}}=\underline{E}C\td{g_1, \dots, g_n}$ by Corollary \ref{Normalizer and Centralizer}. \end{proof}

\begin{rem} Proposition \ref{chromaticidentity} and Theorem \ref{theorema} are related to the decomposition of the iterated $p$-adic free loop space due to Lurie \cite[Example 3.4.5]{Elliptic} and \cite{LPS}. We do not need this result in the current paper but since it offers an alternative way to prove Theorem \ref{theorema}, we still review it. Let $B_{\gl}G$ denote the global classifying orbispace of $G$ and $(B_{\gl}G)^{\wedge}_{n,p}$ the iterated $p$-adic free loop space of  $B_{\gl}G$. If $G$ admits a finite $\un{E}G$, then one has a splitting of orbispaces
\[  \coprod_{[g_1, \dots, g_n] \in G \bs G_{n,p}} B_{\gl}C\td{g_1, \dots, g_n}\ \xra{\ \simeq \ } \ (B_{\gl}G)^{\wedge}_{n,p}. \  \] 
For finite groups, this follows from \cite[Example 3.4.5]{Elliptic} and, for the general case, from the forthcoming paper \cite{LPS}. By applying to the splitting the orbispace Euler characteristic (see \cite{LPS}), we get 
\[\chi_{\orb}[(B_{\gl}G)^{\wedge}_{n,p}] = \sum_{[g_1,\dots,g_n] \in G \bs G_{n,p}} \chi_{\orb}[B_{\gl}C\td{g_1, \dots, g_n}]=\sum_{[g_1,\dots,g_n] \in G \bs G_{n,p}} \chi_{\orb}(C\td{g_1, \dots, g_n}).\]
Using this and the class equation, for any finite group $H$, one obtains
\[\chi_{\orb}[(B_{\gl}H)^{\wedge}_{n,p}]=\sum_{[h_1,\dots, h_n] \in H \bs H_{n,p}} \frac{1}{\vert C\td{h_1, \dots, h_n}\vert} =  \frac{\vert H_{n,p}\vert} {\vert H\vert}.\]
On the other hand, since the iterated $p$-adic free loop space commutes with colimits, we also get
\[\chi_{\orb} [(B_{\gl}G)^{\wedge}_{n,p}]=\sum_{G\sigma} (-1)^{n_\sigma} \frac{\vert H^{\sigma}_{n,p}\vert} {\vert H^{\sigma} \vert},\]
where the sum runs over all $G$-orbits of cells $\sigma$ of $\un{E}G$, the number $n_\sigma$ is the dimension of $\sigma$, and $H^{\sigma}$ is the stabilizer of $\sigma$. Combining these results, we get an alternative proof of the identity 
\[ \sum_{[g_1,\dots,g_n] \in G \bs G_{n,p}} \chi_{\orb}(C\td{g_1, \dots, g_n})=\sum_{G\sigma} (-1)^{n_\sigma} \frac{\vert H^{\sigma}_{n,p}\vert} {\vert H^{\sigma} \vert}.\]
Since we want to keep the exposition simple, we will not use orbispaces, $p$-adic free loop spaces, and the orbispace Euler characteristic in this paper. Interested readers are referred to the forthcoming paper \cite{LPS} and \cite{Elliptic}. 

\end{rem}

\section{Generating tuples and the proof of Theorem \ref{theoremb}}

In this section we prove Theorem \ref{theoremb}. The strategy is to use generating tuples. Let $H$ be a finite abelian subgroup of $G$. Define
\[\Gen_n(H)=\{ (h_1, \dots, h_n) \in H^{\times n} \; \vert \; \td{h_1, \dots, h_n}=H\}.\]
In other words, $\Gen_n(H)$ is the set of generating $n$-tuples of $H$. The Weyl group $W(H)=N(H)/C(H)$ acts freely on $\Gen_n(H)$ by conjugation in each coordinate. 

\begin{lemma} \label{generating tuples formula} Let $G$ be a virtually torsion-free discrete group with a finite $\underline{E}G$. Then we have
\[\sum_{[g_1,\dots,g_n] \in G \bs G_{n,p}} \chi_{\orb}(C\td{g_1, \dots, g_n}) =\sum_{(H)} \vert \Gen_n(H) \vert \chi_{\orb} (N(H)),\]
where the right hand sum runs over the conjugacy classes of finite abelian $p$-subgroups. 
\end{lemma}

\begin{proof} 
It follows by definition that the set $G_{n,p}$ admits a decomposition 
\[\coprod_{\tiny H \text{ab}. p\text{-subg.}} \Gen_n(H) = G_{n,p},\]
where the disjoint union runs over all finite abelian $p$-subgroups. The $G$-action on the right hand side corresponds to a $G$-action on the left hand side that is given by conjugating tuples and permuting summands via conjugation of subgroups. Passing to the $G$-quotients gives a bijection
\[\coprod_{(H)} W(H)\bs \Gen_n(H) \cong G \bs G_{n,p},\]
where the disjoint union runs over the conjugacy classes of finite abelian $p$-subgroups. This implies 
\[\sum_{[g_1,\dots,g_n] \in G \bs G_{n,p}} \chi_{\orb}(C\td{g_1, \dots, g_n})=\sum_{(H)} \vert W(H)\bs \Gen_n(H) \vert \chi_{\orb} (C(H)).\]
Since $W(H)$ acts freely on $\Gen_n(H)$ and $\chi_{\orb}(N(H))=\frac{1}{\vert W(H) \vert} \chi_{\orb} (C(H))$, we get the desired equation. \end{proof}

Let $H$ be a finite abelian $p$-group. The Frattini subgroup $\Phi(H)$ is the smallest subgroup such that $H/\Phi(H)$ is elementary abelian. 
Note that $\Phi(H)=1$ if and only if $H$ is an elementary abelian $p$-group. For $E$ an elementary abelian $p$-group, the rank of $E$ is denoted by $r(E)$. We also recall the definition of the $p$-binomial coefficients:
\[\binom{B}{i}_p=\frac{(p^B-1) \dots (p^{B-i+1}-1) }{(p^i-1) \dots (p-1)}.\]
In particular, $\binom{B}{B}_p=1$ and  $\binom{B}{0}_p=1$. We are now ready to prove Theorem \ref{theoremb}. 

\begin{theorem} \label{generating tuples and congruences}  Let $G$ be a virtually torsion-free discrete group with a finite $\underline{E}G$. Then we have
\begin{align*}\sum_{[g_1,\dots,g_n] \in G \bs G_{n,p}} \chi_{\orb}(C\td{g_1, \dots, g_n})\hspace{5cm}\\= \sum_{(H)} \chi_{\orb} (N(H)) \cdot \vert  \Phi(H) \vert^n \cdot \big ( \sum_{i=0}^{r(H/\Phi(H))} (-1)^i p^{n(r(H/\Phi(H))-i)+\binom{i}{2}} \cdot \binom{r(H/\Phi(H))}{i}_p  \big ), \end{align*}
where the sum runs over the conjugacy classes of finite abelian $p$-subgroups. Furthermore, we have the congruence
\[\sum_{(H)} \chi_{\orb} (N(H)) \cdot  \vert \Phi(H)\vert^n \cdot \big ( \sum_{i=0}^{r(H/\Phi(H))} (-1)^i p^{n(r(H/\Phi(H))-i)+\binom{i}{2}} \cdot \binom{r(H/\Phi(H))}{i}_p  \big) \equiv 0 \mod \Z_{(p)}.\]
\end{theorem}

\begin{proof} Using \cite[Section 6.B]{Dia} and \cite{Hall}, we know the cardinality of the set of generating $n$-tuples:
\[\vert \Gen_n(H) \vert= \vert \Phi(H) \vert^n \cdot \big ( \sum_{i=0}^{r(H/\Phi(H))} (-1)^i p^{n(r(H/\Phi(H))-i)+\binom{i}{2}} \cdot \binom{r(H/\Phi(H))}{i}_p  \big). \]
By Lemma \ref{generating tuples formula}, we get the first claim. The congruence now follows from the first claim and Theorem \ref{theorema}. \end{proof}

\begin{rem} \label{max card} The formula of Theorem \ref{generating tuples and congruences} can be slightly simplified. In fact it suffices to take the sum indexed over conjugacy classes of those finite abelian $p$-subgroups which are generated by at most $n$ elements. Other terms in the sum will automatically vanish. For example, if $n=1$, then we only need to consider cyclic $p$-subgroups. This agrees with Theorem \ref{theoremb}. 
\end{rem}

\begin{rem} It is sometimes convenient to rewrite the congruences of Theorem \ref{theorema} and Theorem \ref{theoremb} as follows:
\[\chi_{\orb}(G) \equiv \sum_{[g_1,\dots,g_n] \in G \bs G_{n,p}-{\{[1,\dots,1]\}}} -\chi_{\orb}(C\td{g_1, \dots, g_n}) \mod \Z_{(p)}\]
and 
\begin{align*}\chi_{\orb}(G) \equiv  \sum_{(H), H \neq 1} \chi_{\orb} (N(H)) \cdot  \vert \Phi(H) \vert^n \cdot \big ( \sum_{i=0}^{r(H/\Phi(H))} (-1)^{i+1} p^{n(r(H/\Phi(H))-i)+\binom{i}{2}} \cdot \binom{r(H/\Phi(H))}{i}_p  \big) \\ \mod \Z_{(p)}.\end{align*} 
\end{rem}

\section{The congruence at height $n=\infty$ and a new proof of the Brown-Quillen congruence} 

In this section we give a new proof of the following theorem of Brown and Quillen \cite[Theorem 14.2]{Brownbook}:

\begin{theorem} \label{Brown-Quillen}  Let $G$ be a virtually torsion-free discrete group with a finite $\underline{E}G$. Then one has
\[ \sum_{(E)} (-1)^{r(E)}p^{\binom{r(E)}{2}}\chi_{\orb}(N(E)) \equiv 0 \mod \Z_{(p)},\]
where the sum runs over the conjugacy classes of elementary abelian $p$-subgroups.
\end{theorem}


\begin{proof} For $H=\prod_{i=0}^B \Z/p^{\lambda_i}$ we have: $\Phi(H)=\prod_{i=1}^B \Z/p^{\lambda_i-1}$, $B=r(H/\Phi(H))$, and $ \vert \Phi(H)\vert=p^{A-B}$, where $A=\sum_{i=0}^B \lambda_i$. We keep in mind that $A$ and $B$ depend on $H$. Then by Theorem \ref{generating tuples and congruences}, we have the congruence
\[\sum_{(H)} \chi_{\orb} (N(H)) \cdot  p^{n(A-B)} \cdot \big ( \sum_{i=0}^{B} (-1)^i p^{n(B-i)+\binom{i}{2}} \cdot \binom{B}{i}_p  \big) \equiv 0 \mod \Z_{(p)}.\]
Let $a_n$ denote the left hand side of this congruence. Then we know $a_n \in \Z_{(p)}$. By taking the $p$-adic limit, we get
\begin{align*}&\lim _{n\to \infty}a_n=\sum_{(H)} \chi_{\orb} (N(H)) \cdot  \lim_{n\to \infty}(p^{n(A-B)} \cdot \big ( \sum_{i=0}^{B} (-1)^i p^{n(B-i)+\binom{i}{2}} \cdot \binom{B}{i}_p  \big)) \\ &= \sum_{(H)} \chi_{\orb} (N(H)) \cdot  \lim_{n\to \infty}p^{n(A-B)} \cdot \big ( \sum_{i=0}^{B} (-1)^i \lim_{n \to \infty} p^{n(B-i)+\binom{i}{2}} \cdot \binom{B}{i}_p  \big).\end{align*}
The terms with $A>B$ or $B>i$ vanish since $\lim_{n \to \infty} p^n=0$ in $\Z_p$. An abelian $p$-subgroup is elementary abelian if and only if $A=B$, i.e., when the Frattini subgroup $\Phi(H)$ is trivial. Hence only $A=B$ and $i=B$ terms survive in the limit and we obtain
\[\lim_{n \to \infty} a_n=\sum_{(E)} (-1)^{r(E)}p^{\binom{r(E)}{2}}\chi_{\orb}(N(E)),\]
where the sum runs over the conjugacy classes of elementary abelian $p$-subgroups. The latter is a rational number, so we know $\lim_{n \to \infty} a_n \in \mathbb{Q}$. On the other hand $\Z_p$ is closed in $\mathbb{Q}_p$ and since $a_n \in \Z_{(p)} \leq \Z_p$, we know that the limit $\lim_{n \to \infty} a_n$ also belongs to $\Z_p$. Hence 
$\lim_{n \to \infty} a_n$ belongs to the intersection $\mathbb{Q} \cap \Z_p= \Z_{(p)}$ and we obtain the desired congruence. 
\end{proof}

\begin{example} Suppose $G$ satisfies the assumptions of Theorem \ref{Brown-Quillen} and additionally assume that every elementary abelian $p$-subgroup of $G$ has the rank at most $1$. Then by Lemma \ref{generating tuples formula} (see also \cite[Corollary 13.5]{Brownbook}), one has
\[\sum_{(P)} \chi_{\orb}(N(P))=\frac{1}{p-1} \sum_{[g]} \chi_{\orb}(C\td{g}),\]
where $P$ runs over the conjugacy classes of order $p$ subgroups and $g$ runs over the conjugacy classes of order $p$ elements. This shows that Brown's congruence mentioned in the introduction is a special case of the Brown-Quillen congruence. \end{example}

\begin{rem} \label{limit attained} Let $G$ be a virtually torsion-free discrete group with a finite $\underline{E}G$ and suppose $p^N$ is the maximal cardinality of a $p$-subgroup. Then the Brown-Quillen congruence can be recovered from the chromatic congruence at height $N$. In other words, the limit congruence of Theorem \ref{Brown-Quillen} is attained for sufficiently large $N$. Indeed, by Theorem \ref{theoremb} it suffices to show that the term 
\[ \chi_{\orb} (N(H)) \cdot  \vert \Phi(H) \vert^N \cdot p^{N(r(H/\Phi(H))-i)+\binom{i}{2}} \cdot \binom{r(H/\Phi(H))}{i}_p  \]
is $p$-integral if $H$ is not elementary abelian or $i<r(H/\Phi(H))$. But this follows since $p$-binomial coefficients are $p$-integral and $p^N \chi_{\orb} (N(H))$ is $p$-integral by Quillen's formula (Proposition \ref{Quillenformula}). The remaining terms in the congruence of Theorem \ref{theoremb} recover the Brown-Quillen congruence. \end{rem}

\section{An application: Mapping class groups and congruences for Bernoulli numbers}

In this section we apply Theorem \ref{theorema} to the mapping class group $\Gamma_u$ of a closed oriented surface $S_u$ of genus $u$ and recover special cases of Kummer's congruence \cite{Kum, Vor}, Carlitz's congruence \cite{Carl} and a lesser known version of Kummer's congruence for numbers divisible by $p-1$ \cite[Proposition 11.4.4]{Cohen} 

By \cite{HZ}, one has 
\[\chi_{\orb}(\Gamma_u)=\frac{\zeta(1-2u)}{2-2u}=\frac{B_{2u}}{2u(2u-2)},\]
where $\zeta$ is the Riemann zeta function and $B_{2u}$ is the $2u$-th Bernoulli number. Let $\Gamma_u^s$ denote the mapping class group of a closed oriented surface of genus $u$ with $s$ many marked points (in particular, one has $\Gamma_u^0=\Gamma_u$). We need the following formulas from \cite{HZ}:
\[\chi_{\orb}(\Gamma^s_0) = \begin{cases} 1  \;\;\;\; s \leq 3 \\ (-1)^{s-3}(s-3)! \;\;\;\; s \geq 3, 
\end{cases} \]
\[\chi_{\orb}(\Gamma^s_1) = \begin{cases} -\frac{1}{12}  \;\;\;\; s \leq 1 \\ \frac{(-1)^{s}(s-1)!}{12} \;\;\;\; s \geq 1, 
\end{cases} \]
\[\chi_{\orb}(\Gamma^s_u)=(-1)^s\frac{(2u+s-3)!}{2u(2u-2)!}B_{2u} \;\;\;\;  u \geq 2, \;\;\;\; s \geq 0. \]
By \cite{Broughton, Har, Mislin, JiWol}, the group $\Gamma_u^s$ is virtually torsion free and admits a finite $\underline{E}G$. Brown's theorem \cite{KBro2} tells us that for any virtually torsion-free group $G$ with a finite $\underline{E}G$, the equation holds
\[\chi_{\mathbb{Q}}(G)=\sum_{[g]} \chi_{\orb} (C\td{g}),\]
where $\chi_{\mathbb{Q}}(G)$ is the classical Euler characteristic (and hence an integer) and the sum runs over the conjugacy classes of finite order elements of $G$. Harer and Zagier use this equation in \cite[Theorem 5]{HZ} to compute $\chi_{\mathbb{Q}}(\Gamma_u)$ explicitly:
\[\chi_{\mathbb{Q}}(\Gamma_u)=\sum_{\tiny{\begin{aligned} k \geq 1 ,\;\;v \geq 0,\;\;s \geq 0 \hspace{0.75cm} \\ (l_1,\dots,l_s), \;\; l_i \; \vert \; k, \;\; l_i \neq k \hspace{0.5cm}  \\2u-2=k(2v-2+s)-l_1-\dots -l_s  \end{aligned}}} \frac{1}{k}\frac{\chi_{\orb}(\Gamma_v^s)}{s!}k^{2v}\prod_{\tiny{\begin{aligned} &q \; \vert \; \gcd(l_1, \dots l_s)\\ &q \; \vert \; k, q \; \text{prime}  \end{aligned} }} (1-q^{-2v})N(k;l_1,\dots,l_s),\]
where $N(k;l_1,\dots,l_s)$ is the cardinality of the set
\[\{ (r_1, \dots r_s) \in (\Z/k\Z)^{\times s} \;\; \vert \;\; r_1 +\dots+r_s \equiv 0 \mod k, \; \gcd(r_i,k)=l_i\}.\]
The number $k$ stands for the order of elements. In this formula, the sum of the terms with a fixed $k \geq 0$ is equal to the sum $\sum_{[g]} \chi_{\orb} (C\td{g})$, where $[g]$ runs over the conjugacy classes of order $k$ elements. The equation 
\[2u-2=k(2v-2+s)-l_1-\dots -l_s \]
is the Riemann-Hurwitz formula for a branched cover associated to an order $k$ element of $\Gamma_u$ with $s$ many singular points. This cover can be constructed using the Nielsen realization \cite{Nielsen}: Given an element $g$ of order $k$ in $\Gamma_u$, the mapping class $g$ is represented by a periodic homeomorphism $f$ of order $k$. The associated branched cover is the map $S_u \to S_u/\td{f}$. The $i$-th singularity has type $k/l_i$ and the expression 
\[k^{2v}\prod_{\tiny{\begin{aligned} & q \; \vert \; \gcd(l_1, \dots l_s)\\ & q \; \vert \; k, q \; \text{prime} \end{aligned} }} (1-q^{-2v})\] 
calculates the number of surjective characters determining the cover on the non-singular part. 

Let $p$ be a prime. If we only allow $k$ to run through the powers of $p$, using the Harer-Zagier formula above and Theorem \ref{theorema} at height $n=1$, we obtain

\begin{prop} \label{p-localHZ} For any $u \geq 0$, we have the congruence
\begin{align*}\chi_{\orb}(\Gamma_u)+\sum_{\tiny{\begin{aligned} m \geq 1 ,\;\;v \geq 0,\;\;s \geq 0 \hspace{0.75cm} \\ (m_1,\dots,m_s), \;\; 0 \leq m_i < m, \;\; \forall m_i \neq 0 \hspace{0.5cm}  \\2u-2=p^m(2v-2+s)-p^{m_1}-\dots -p^{m_s}  \end{aligned}}} \frac{1}{p^m}\frac{\chi_{\orb}(\Gamma_v^s)}{s!}p^{2mv}(1-p^{-2v})N(p^m;p^{m_1},\dots,p^{m_s})\\+\sum_{\tiny{\begin{aligned} m \geq 1 ,\;\;v \geq 0,\;\;s \geq 0 \hspace{0.75cm} \\ (m_1,\dots,m_s), \;\; 0 \leq m_i < m, \;\; \exists m_i = 0 \hspace{0.5cm}  \\2u-2=p^m(2v-2+s)-p^{m_1}-\dots -p^{m_s}  \end{aligned}}} \frac{1}{p^m}\frac{\chi_{\orb}(\Gamma_v^s)}{s!}p^{2mv}N(p^m;p^{m_1},\dots,p^{m_s}) \equiv 0 \mod \Z_{(p)}.\end{align*}\qed \end{prop}

\noindent To extract concrete number theoretic formulas from this congruence, we need to simplify the left hand side in $\mathbb{Q}/\Z_{(p)}$. It turns out that most of the terms in this congruence are already $p$-integral. Though the calculations are somewhat lengthy, we give here detailed proofs. The only number theoretic input used below is the von Staudt-Clausen theorem \cite{Staudt, Clausen} (see also \cite[Section 3.1]{Arak}) or rather its $p$-local consequence: Let $p$ be a prime and suppose $(p-1) \; \vert \; 2u$. Then 
\[B_{2u}+\frac{1}{p} \equiv 0 \mod \Z_{(p)}.\] 
For the rest of the section we will assume that the prime $p$ is at most $5$ and $u \geq 2$.  We start with 

\begin{lemma} \label{mhgeq2} Let $m \geq 2$, $v \geq 2$, $s \geq 0$, $0 \leq m_i < m$, and $2u-2=p^m(2v-2+s)-p^{m_1}-\dots -p^{m_s}$. 
Then the expression
\[\frac{1}{p^m}\frac{\chi_{\orb}(\Gamma_v^s)}{s!}p^{2mv-2v}N(p^m;p^{m_1},\dots,p^{m_s})\]
is $p$-integral. 
\end{lemma} 

\begin{proof} The expression can be rewritten as follows:
\begin{align*}\frac{1}{p^m}\frac{\chi_{\orb}(\Gamma_v^s)}{s!}p^{2mv-2v}N(p^m;p^{m_1},\dots,p^{m_s})=\frac{1}{p^m}\cdot \frac{(-1)^s(2v+s-3)!}{s! \cdot 2v(2v-2)!}B_{2v}p^{2mv-2v}N(p^m;p^{m_1},\dots,p^{m_s})\\=\frac{B_{2v}}{p^m \cdot 2v(2v-2)} p^{2mv-2v} \cdot (-1)^s \frac{(2v+s-3)!}{s!(2v-3)!} N(p^m;p^{m_1},\dots,p^{m_s})=\frac{B_{2v}}{p^m \cdot 2v(2v-2)} p^{2mv-2v} \cdot M,\end{align*}
where $M$ is an integer. Now using the von Staudt-Clausen theorem, it suffices to show
\[2mv-2v \geq m+1 +\nu_p(2v(2v-2)),\]
where $\nu_p$ is the $p$-adic valuation. This follows immediately since $v-1 \geq \nu_p(2v(2v-2))$ and 
\[2mv-2v \geq m+v,\]
for $m \geq 2$ and $v \geq 2$. \end{proof}

The following deals with the case $m=1, v \geq 2$ and is proved analogously as the previous lemma:

\begin{lemma} \label{m1hgeq2} Let $v \geq 2$, $s > 0$, and $2u-2=p(2v-2+s)-s$. Then the expression 
\[\frac{1}{p}\frac{\chi_{\orb}(\Gamma_v^s)}{s!}p^{2v}N(p;\underbrace{1, \dots, 1}_{s})\]
is $p$-integral. 
\end{lemma}


The above two lemmas say that all the terms in the formula of Proposition \ref{p-localHZ} are $p$-integral except possibly the terms with $m=1, s=0$, or $m \geq 1, v=0$, or $m \geq 1, v=1$. In these cases many of the terms are still $p$-integral. To be able to deal with these expressions we need an explicit formula for the number $N(p^m;p^{m_1},\dots,p^{m_s})$. The following follows from \cite[page 81]{HZ}:

\begin{lemma} \label{Nnumber} Let $m \geq 1$, $s \geq 0$, $0 \leq m_i < m$. Then
\[N(p^m;p^{m_1},\dots,p^{m_s})=\frac{1}{p^m} \prod_{i=1}^s(p^{m-m_i}-p^{m-m_i-1})p^{\lambda_p} \Big (1-\frac{(-1)^{\mu_p-1}}{(p-1)^{\mu_p-1}} \Big ),\] 
where $\lambda_p=\min\{m_1, \dots, m_s, m-1\}$, and $\mu_p$ is the cardinality of the set
\[\{\; m_i \; \vert \; p^{\lambda_p+1}  \nmid p^{m_i} \; \}.\]
In particular, we have
\[N(p, \underbrace{1, \dots, 1}_{s})=\frac{1}{p}(p-1)^s \Big (1-\frac{(-1)^{s-1}}{(p-1)^{s-1}} \Big ).\] \qed
\end{lemma}

\begin{rem} \label{mup=1} There is a typo on page 81 of \cite{HZ}, stating that in the above formula one has the factor $\Big (1-\frac{(-1)^{\mu_p}}{(p-1)^{\mu_p}} \Big )$ instead of $\Big(1-\frac{(-1)^{\mu_p-1}}{(p-1)^{\mu_p-1}} \Big )$. In particular, when $\mu_p=1$, then $N(p^m;p^{m_1},\dots,p^{m_s})=0$. \end{rem} 

The sums in Proposition \ref{p-localHZ} are indexed on $s$-tuples $(m_1, \dots, m_s)$. We note that any permutations of such tuples are allowed and the terms involved are independent of these permutations. Let $\pi(m_1,\dots, m_s)$ denote the number of different tuples obtained by permuting $(m_1, \dots, m_s)$. This can be described as a multinomial coefficient. Every term in Proposition \ref{p-localHZ} indexed on $(m_1, \dots, m_s)$ occurs  $\pi(m_1,\dots, m_s)$ many times. This allows us to prove further $p$-integrality statements for the terms involved in Proposition \ref{p-localHZ}. 

\;

We deal next with the case $m \geq 1, v=1$. 

\begin{lemma} \label{oneminot m-1} Let $m \geq 1$, $s \geq 0$, $0 \leq m_i < m$, $2u-2=p^ms-p^{m_1}-\dots -p^{m_s}$ and suppose there exists an index $i_0$ such that $m_{i_0} \neq m-1$. Then the expression
\[\frac{1}{p^m}\frac{\chi_{\orb}(\Gamma_1^s)}{s!}p^{2m-2}N(p^m;p^{m_1},\dots,p^{m_s})\pi(m_1, \dots, m_s)\]
is $p$-integral. 
\end{lemma}

\begin{proof} It follows that $m \geq 2$. Let $x$ be the number of those $i$-s such that $m_i$ is minimal and $m_i \neq m-1$. Then we have $x \geq 1$. If $x=1$, then $N(p^m;p^{m_1},\dots,p^{m_s})=0$ by Remark \ref{mup=1} and the claim follows. We may assume that $x$ is at least $2$. By definition of  $\pi(m_1,\dots, m_s)$ and the formula for multinomial coefficients, we get that
\[\pi(m_1,\dots, m_s)=\frac{s}{x} \cdot M,\]
where $M$ is an integer (in fact a multinomial coefficient with $(s-1)!$ as the numerator and involving $(x-1)!$ in the denominator). On the other hand by Lemma \ref{Nnumber}, we get
\[N(p^m;p^{m_1},\dots,p^{m_s})=\frac{1}{p^m} \prod_{i=1}^s p^{m-m_i-1} L'= \frac{1}{p^m} p^x \cdot L,\]
where $L'$ and $L$ are $p$-integers. Hence the relevant expression can be rewritten as follows
\begin{align*} &\frac{1}{p^m}\frac{\chi_{\orb}(\Gamma_1^s)}{s!}p^{2m-2}N(p^m;p^{m_1},\dots,p^{m_s})\pi(m_1, \dots, m_s)\\ &=\frac{1}{p^m} \frac{(-1)^s(s-1)!}{12s!} p^{2m-2} \cdot \frac{1}{p^m}p^x \cdot \frac{s}{x} M L =(-1)^s\cdot\frac{p^{x-2}}{12x} \cdot M \cdot L. \end{align*}
If $x=2$, then this number is $p$-integral since $p \geq 5$. If $x \geq 3$, then using $p \geq 5$, we get $p^{x-2} \geq 5^{x-2} \geq x$, implying that $x-2 \geq \nu_p(x)$. Hence $(-1)^s \cdot \frac{p^{x-2}}{12x} \cdot M \cdot L$ is $p$-integral which completes the proof. \end{proof}

The remaining expressions with $m \geq 1$ and $v=1$ are the terms where $m_i=m-1$ for all $i$. The Riemann-Hurwitz formula $2u-2=p^ms-p^{m-1}s$ tells us that $s$ is uniquely determined by $m$ in such a case. To make this clear we introduce the notation $s(m)=\frac{2u-2}{p^m-p^{m-1}}$. Let $r$ be the largest number such that $(p^{r+1}-p^r) \; \vert \; 2u-2$. Then we have $2u-2=(p^m-p^{m-1})s(m)$ for $1 \leq m \leq r+1$. 

\begin{lemma} \label{allmim-1} Let $p-1 \; \vert \; 2u-2$ and $r \geq 0$ be the largest number such that $(p^{r+1}-p^r) \; \vert \; 2u-2$. Then
\[\frac{1}{p}\frac{\chi_{\orb}(\Gamma_1^{s(1)})}{s(1)!}p^{2}N(p;\underbrace{1,\dots,1}_{s(1)}) + \sum_{m =2}^{r+1} \frac{1}{p^m}\frac{\chi_{\orb}(\Gamma_1^{s(m)})}{s(m)!}p^{2m}(1-p^{-2})N(p^m;\underbrace{p^{m-1},\dots,p^{m-1}}_{s(m)}) \equiv 0 \mod \Z_{(p)}.\]
\end{lemma}

\begin{proof} We first simplify the separate summands in $\mathbb{Q}/\Z_{(p)}$. By Lemma \ref{Nnumber} and the binomial theorem for any $2 \leq m \leq r+1$, one has
\begin{align*} &-\frac{1}{p^m}\frac{\chi_{\orb}(\Gamma_1^{s(m)})}{s(m)!}p^{2m-2}N(p^m;\underbrace{p^{m-1},\dots,p^{m-1}}_{s(m)})\\ &=-\frac{(-1)^{s(m)}p^{m-2}}{12s(m)} \cdot \frac{1}{p}(p-1)^{s(m)} \Big (1-\frac{(-1)^{s(m)-1}}{(p-1)^{s(m)-1}} \Big )\\&=-\frac{(-1)^{s(m)}p^{m-3}}{12s(m)} \Big ( (p-1)^{s(m)} +(-1)^{s(m)}(p-1)  \Big )\\&= -\frac{(-1)^{s(m)}p^{m-3}}{12s(m)} \Big (\sum_{x=0}^{s(m)} \Big (p^x(-1)^{s(m)-x} \frac{s(m)!}{x!(s(m)-x)!} \Big ) +(-1)^{s(m)}(p-1) \Big ) \\&=-\sum_{x=1}^{s(m)} \frac{(-1)^{2s(m)-x} p^{x+m-3}} {12x} \binom{s(m)-1}{x-1} -\frac{p^{m-2}}{12s(m)} \equiv -\frac{p^{m-2}}{12s(m)}  \mod \Z_{(p)}.  \end{align*}
The latter congruence comes from the fact that since $m \geq 2$ and $x \geq 1$, we have $p^{x+m-3} \geq x$. A very similar argument also shows that for any $1 \leq m \leq r+1$, we have
\[\frac{1}{p^m}\frac{\chi_{\orb}(\Gamma_1^{s(m)})}{s(m)!}p^{2m}N(p^m;\underbrace{p^{m-1},\dots,p^{m-1}}_{s(m)}) \equiv \frac{p^m}{12s(m)} \mod \Z_{(p)}.\]
We conclude that the left hand side of the the desired congruence is equal in $\mathbb{Q}/\Z_{(p)}$ to the sum
\[\frac{p}{12s(1)}+\sum_{m=2}^{r+1} \Big (\frac{p^m}{12s(m)} -\frac{p^{m-2}}{12s(m)}  \Big ).\]
Using that $ps(m)=s(m-1)$ for any $2 \leq m \leq r+1$, the latter sum is equal to $\frac{p^{r+1}}{12s(r+1)}$ in $\mathbb{Q}/\Z_{(p)}$. 
By definition of $r$, we know that $p \nmid s(r+1)$. Since $p \geq 5$, this implies that $\frac{p^{r+1}}{12s(r+1)}$ is $p$-integral which completes the proof. \end{proof}

Finally, we deal with the case $m \geq 1, v=0$. We start with the following:
\begin{lemma} \label{3mi=0} Let $m \geq 2$, $s \geq 0$, $0 \leq m_i < m$, and $2u-2=p^m(s-2)-p^{m_1}-\dots -p^{m_s}$. Suppose that the cardinality of the set $\{ \; i \; \vert \; m_i=0 \;\}$ is at least $3$. Then the expression
\[ \frac{1}{p^m}\frac{\chi_{\orb}(\Gamma_0^s)}{s!}N(p^m;p^{m_1},\dots,p^{m_s})\pi(m_1, \dots, m_s)\]
is $p$-integral. \end{lemma} 

\begin{proof} Let $x$ denote the cardinality of the set $\{ \; i \; \vert \; m_i=0 \;\}$. Then by our assumption, we have $x \geq 3$ and hence also $s \geq 3$. 
By Lemma \ref{Nnumber}, one gets
\begin{align*} &  \frac{1}{p^m}\frac{\chi_{\orb}(\Gamma_0^s)}{s!}N(p^m;p^{m_1},\dots,p^{m_s})\pi(m_1, \dots, m_s)\\&= \frac{(-1)^{s-3}}{p^ms(s-1)(s-2)} \cdot \frac{1}{p^m} \prod_{i=1}^s p^{m-m_i-1} (p-1)^s \Big ( 1-\frac{(-1)^{x-1}}{(p-1)^{x-1}}  \Big) \pi(m_1, \dots, m_s).\\ \end{align*}
The number $\pi(m_1, \dots, m_s)$ can be expressed as a multinomial coefficient. Since $x \geq 3$, it can be written as follows
\[\pi(m_1, \dots, m_s)=\frac{s(s-1)(s-2)}{x(x-1)(x-2)} \cdot \frac{(s-3)!}{(x-3)! \cdots} = \frac{s(s-1)(s-2)}{x(x-1)(x-2)} \cdot M,\]
where $M$ is an integer. We also know that $p^{(m-1)x} \; \vert \;  \prod_{i=1}^s p^{m-m_i-1}$ and $p \; \vert \;((p-1)^{x-1} +(-1)^x)$. Hence we get 
\begin{align*} \frac{1}{p^m}\frac{\chi_{\orb}(\Gamma_0^s)}{s!}N(p^m;p^{m_1},\dots,p^{m_s})\pi(m_1, \dots, m_s)=\frac{(-1)^{s-3}p^{(m-1)x-2m+1}}{x(x-1)(x-2)} \cdot L,  \end{align*}
where $L$ is $p$-integral. If $x=3$ then the latter is $p$-integral since $m \geq 2$ and $p \geq 5$. And if $x \geq 4$, then $p^{(m-1)x-2m+1} \geq x$ since $p \geq 5$ and $m \geq 2$, thus showing that $\frac{(-1)^{s-3}p^{(m-1)x-2m+1}}{x(x-1)(x-2)}$ is $p$-integral. This completes the proof. \end{proof}

The next lemma is proved analogously as the previous one and we skip the proof: 

\begin{lemma} \label{2mo=0onelessm-1} Let $m \geq 2$, $s \geq 0$, $0 \leq m_i < m$, and $2u-2=p^m(s-2)-p^{m_1}-\dots -p^{m_s}$. Suppose that the cardinality of the set $\{ \; i \; \vert \; m_i=0 \;\}$ is equal to $2$ and there exists an index $j_0$ such that $0<m_{j_0} <m-1$. Then the expression
\[ \frac{1}{p^m}\frac{\chi_{\orb}(\Gamma_0^s)}{s!}N(p^m;p^{m_1},\dots,p^{m_s})\pi(m_1, \dots, m_s)\]
is $p$-integral. \end{lemma}


Finally, we consider the case when $m \geq 1$, $v=0$, the cardinality of the set $\{ \; i \; \vert \; m_i=0 \;\}$ is equal to $2$ and all the other $m_i$-s are equal to $m-1$.  The Riemann-Hurwitz formula $2u=(p^m-p^{m-1})(s-2)$ tells us that $s$ is uniquely determined by $m$. To make this clear we introduce the notation $s(m)=\frac{2u}{p^m-p^{m-1}}+2$. Let $r$ be the largest number such that $(p^{r+1}-p^r) \;  \vert  \; 2u$. Then we have $2u=(p^m-p^{m-1})(s(m)-2)$ for $1 \leq m \leq r+1$. 

\begin{lemma} \label{2mi0otherm-1} Let $p-1 \; \vert \; 2u$ and $r \geq 0$ be the largest number such that $(p^{r+1}-p^r) \; \vert \; 2u$. Then 
\begin{align*}\frac{1}{p}\frac{\chi_{\orb}(\Gamma_0^{s(1)})}{s(1)!}N(p;\underbrace{1,\dots,1}_{s(1)}) + \sum_{m =2}^{r+1} \frac{1}{p^m}\frac{\chi_{\orb}(\Gamma_0^{s(m)})}{s(m)!}N(p^m;\underbrace{p^{m-1},\dots,p^{m-1}}_{s(m)-2},1,1) \pi(\underbrace{m-1,\dots,m-1}_{s(m)-2},0,0) \\ \equiv \frac{1}{ps(1)(s(1)-2)}  \mod \Z_{(p)}.\hspace{4.5cm}\end{align*}\end{lemma}

\begin{proof} We first simplify the separate summands in $\mathbb{Q}/\Z_{(p)}$. By Lemma \ref{Nnumber} and the binomial theorem, one has
\begin{align*}& \frac{1}{p}\frac{\chi_{\orb}(\Gamma_0^{s(1)})}{s(1)!}N(p;\underbrace{1,\dots,1}_{s(1)})  = \frac{(-1)^{s(1)-3}}{ps(1)(s(1)-1)(s(1)-2)} \cdot \frac{1}{p} (p-1)^{s(1)} \Big ( 1- \frac{(-1)^{s(1)-1}}{(p-1)^{s(1)-1}}\Big ) \\&=\frac{(-1)^{s(1)-3}}{p^2s(1)(s(1)-1)(s(1)-2)} \Big ( (p-1)^{s(1)} +(-1)^{s(1)} (p-1) \Big)  \\&= \frac{(-1)^{s(1)-3}}{p^2s(1)(s(1)-1)(s(1)-2)} \Big((-1)^{s(1)}p+ (-1)^{s(1)-1}ps(1)\\&+(-1)^{s(1)-2}p^2\cdot \frac{s(1)!}{2!\cdot (s(1)-2)!} + \sum_{x=3}^{s(1)}p^x \cdot (-1)^{s(1)-x}\cdot \frac{s(1)!}{x!\cdot (s(1)-x)!}    \Big ) \\&= \frac{1}{p(s(1)-1)(s(1)-2)}- \frac{1}{ps(1)(s(1)-1)(s(1)-2)}-\frac{1}{2(s(1)-2)}+\sum_{x=3}^{s(1)}\frac{(-1)^{x+3}p^{x-2}}{x(x-1)(x-2)} \binom{s(1)-3}{x-3} \\ & \equiv \frac{1}{p(s(1)-1)(s(1)-2)}- \frac{1}{ps(1)(s(1)-1)(s(1)-2)}-\frac{1}{2(s(1)-2)}\\&= \frac{1}{ps(1)(s(1)-2)}-\frac{1}{2(s(1)-2)} \mod \Z_{(p)}. \hspace{3cm}\end{align*}
Here the congruence comes from the fact that the terms $\frac{p^{x-2}}{x(x-1)(x-2)}$ are $p$-integral for $x \geq 3$. The latter follows since $p \geq 5$, $x \geq 3$, and $p^{x-2} \geq x$. 
Furthermore, again using Lemma \ref{Nnumber} and the binomial theorem, for any $2 \leq m \leq r+1$, a similar calculation shows
\begin{align*} & \frac{1}{p^m}\frac{\chi_{\orb}(\Gamma_0^{s(m)})}{s(m)!}N(p^m;\underbrace{p^{m-1},\dots,p^{m-1}}_{s(m)-2},1,1) \pi(\underbrace{m-1,\dots,m-1}_{s(m)-2},0,0) \\ 
& \equiv \frac{1}{2p(s(m)-2)}-\frac{1}{2(s(m)-2)} \mod \Z_{(p)}.  \end{align*}
Now the left hand side of the desired congruence is equal in $\mathbb{Q}/\Z_{(p)}$ to the sum
\begin{align*}\frac{1}{ps(1)(s(1)-2)}-\frac{1}{2(s(1)-2)} +\sum_{m=2}^{r+1} \Big ( \frac{1}{2p(s(m)-2)}-\frac{1}{2(s(m)-2)}  \Big ).\end{align*}
Using that $p(s(m)-2)=s(m-1)-2$ for any $2 \leq m \leq r+1$, the latter sum simplifies to
\[\frac{1}{ps(1)(s(1)-2)}- \frac{1}{2(s(r+1)-2)},\]
where the term $\frac{1}{2(s(r+1)-2)}$ is $p$-integral by definition of $r$. This completes the proof. \end{proof}

We are finally ready to prove the main theorem of this section:

\begin{theorem} \label{kummer} Let $u \geq 2$ be an integer. Then
\begin{enumerate}

\item If $p-1 \; \vert \; 2u$ and $p \; \vert \; 2u-2$, then
\[ \frac{\zeta(1-2u)}{2-2u}-\frac{1}{p}\cdot \frac{\zeta(1-2v)}{2-2v} + \frac{(p-1)^2}{2u \cdot p(2u+2p-2)} \equiv 0 \mod \Z_{(p)},  \]
where $v=\frac{u-1}{p}+1$. 

\item If $p-1  \nmid  2u$ and $p \; \vert \; 2u-2$, then
\[ \frac{\zeta(1-2u)}{2-2u}-\frac{1}{p}\cdot \frac{\zeta(1-2v)}{2-2v} \equiv 0 \mod \Z_{(p)},\]
where $v=\frac{u-1}{p}+1$. 

\item If $p-1 \; \vert \; 2u$ and $p \nmid 2u-2$, then
\[\frac{\zeta(1-2u)}{2-2u}+\frac{(p-1)^2}{2u \cdot p(2u+2p-2)}\equiv 0 \mod \Z_{(p)}.\]

\item If $p-1 \nmid 2u$ and $p \nmid 2u-2$, then
\[\frac{\zeta(1-2u)}{2-2u} \equiv 0 \mod \Z_{(p)}.\]

\end{enumerate}

\end{theorem}

\begin{proof} All the formulas follow from Proposition \ref{p-localHZ} and the formula $\chi_{\orb}(\Gamma_v)=\frac{\zeta(1-2v)}{2-2v}$ \cite{HZ}, after observing that most of the terms in Proposition \ref{p-localHZ}  are $p$-integral. The terms with $m \geq 2, v \geq 2$, or $m=1, s>0, v \geq 2$, are $p$-integral by Lemma \ref{mhgeq2} and Lemma \ref{m1hgeq2}, respectively. The sum of the terms with $m \geq 1$ and $v=1$ is also $p$-integral by Lemma \ref{oneminot m-1} and Lemma \ref{allmim-1}. To get the desired formulas we need to consider the remaining terms with $m=1, s=0$, or $m \geq 1, v=0$, and go through the four cases:

\begin{enumerate}

\item In this case for the term with $m=1, s=0$, the Riemann-Hurwitz formula gives $2u-2=p(2v-2)$ and hence $v=\frac{u-1}{p}+1$. This term looks as follows
\[ \frac{1}{p}\chi_{\orb}(\Gamma_v)p^{2v}(1-p^{-2v})=\frac{1}{p}\frac{\zeta(1-2v)}{2-2v}p^{2v}-\frac{1}{p} \cdot \frac{\zeta(1-2v)}{2-2v},\]
where the first summand is $p$-integral by the von Staudt-Clausen theorem and since $p^{2v-2} \geq v$. We also have the sum of the terms with $m \geq 1, v=0$. In this sum the terms with $m_i \neq 0$ for all $i$ are zero since $v=0$ and hence $1-p^{-2v}=0$ . The terms with $m_i=0$ for only one $i$ are also zero since in this case $N(p^m, p^{m_1}, \dots, p^{m_s})=0$ by Remark \ref{mup=1}.  Using these observations, Lemma \ref{3mi=0}, Lemma \ref{2mo=0onelessm-1}, Lemma \ref{2mi0otherm-1} and Proposition \ref{p-localHZ}, we get
\[ \chi_{\orb} (\Gamma_u) -\frac{1}{p} \cdot \frac{\zeta(1-2v)}{2-2v} +\frac{1}{ps(s-2)} \equiv 0 \mod \Z_{(p)},\]
where $2u=(s-2)(p-1)$ and hence we get
\[\frac{\zeta(1-2u)}{2-2u}-\frac{1}{p} \cdot \frac{\zeta(1-2v)}{2-2v} +\frac{(p-1)^2}{2u \cdot p(2u+2p-2)}\equiv 0 \mod \Z_{(p)}.\]

\item The terms with $m \geq 1, v=0$ do not occur in this case. Indeed, suppose such a term exists. Then by the Riemann-Huwritz formula, we get
\[2u-2=p^m(s-2)-p^{m_1}-\dots-p^{m_s},\]
where $0 \leq m_i < m$. Hence
\[2u=\sum_{i=1}^s (p^m-p^{m_i})-2(p^m-1)=\sum_{i=1}^s p^{m_i}(p^{m-m_i}-1)-2(p^m-1),\]
implying $p-1 \; \vert \; 2u$, and thus contradicting to our assumption. The term with $m=1, s=0$ is congruent to
\[-\frac{1}{p} \cdot \frac{\zeta(1-2v)}{2-2v} \]
exactly as in Part i). This covers all the summands in Proposition \ref{p-localHZ} and we obtain
\[\frac{\zeta(1-2u)}{2-2u}-\frac{1}{p} \cdot \frac{\zeta(1-2v)}{2-2v} \equiv 0 \mod \Z_{(p)}.\]

\item The term with $m=1, s=0$ does not occur in this case since $p  \nmid 2u-2$ and the Riemann-Hurwitz formula $2u-2=p(2v-2)$ cannot hold. What remains is the sum of the terms with $m \geq 1, v=0$ which is computed exactly as in Part i) and we get
\[\frac{\zeta(1-2u)}{2-2u}+\frac{(p-1)^2}{2u \cdot p(2u+2p-2)}\equiv 0 \mod \Z_{(p)}.\]

\item The term with $m=1, s=0$ does not occur for the same reason as in Part iii) and the terms with $m \geq 1, v=0$ do not occur for the same reason as in Part ii). This covers all the summands and by Proposition \ref{p-localHZ}, we obtain
\[\chi_{\orb}(\Gamma_u)=\frac{\zeta(1-2u)}{2-2u} \equiv 0 \mod \Z_{(p)}.\] 
\end{enumerate}
\end{proof}

\begin{rem} The congruence of Part i) recovers a lesser known version of Kummer's congruence when $p-1 \; \vert \; 2u$ and $p-1 \; \vert \; 2v$ (see \cite[Proposition 11.4.4]{Cohen}). The conditions $p-1 \; \vert \; 2u$ and $p \; \vert \; 2u-2$ hold if and only if the number $u$ can be written as follows
\[u=p^r\Big( \frac{p-1}{2}k-1  \Big)+1,\]
where $k \geq1$ and $r \geq 1$. Then
\[v=p^{r-1}\Big( \frac{p-1}{2}k-1  \Big)+1,\]
and we see $p-1 \; \vert \; 2v$. 

Let $x$ denote the expression $\frac{p-1}{2}k-1$. Then the congruence of Part i) can be rewritten as
\[ -\zeta(1-2u) + \zeta(1-2v) \equiv -\frac{xp^r(p-1)^2}{u \cdot p(2u+2p-2)} \mod p^r\Z_{(p)}. \]
In terms of Bernoulli numbers after simplifying one obtains 
\[\frac{B_{2u}}{2u}- \frac{B_{2v}}{2v} \equiv \Big ( \frac{1}{2u}-\frac{1}{2v} \Big )\Big(1-{\frac{1}{p}}\Big) \mod p^r\Z_{(p)}.\]
This agrees with \cite[Proposition 11.4.4]{Cohen}.


\end{rem}

\begin{rem} Part ii) of Theorem \ref{kummer} recovers a special case of Kummer's congruence proved by Kummer and Voronoi \cite{Kum, Vor}. Indeed, suppose $2u-2=p^ry$, where $\gcd(y,p)=1$. Then $v=\frac{u-1}{p}+1$ and hence $2v-2=p^{r-1}y$. If $p-1 \nmid 2u$, then $p-1 \nmid 2v$. The congruence in Part ii) can be rewritten as follows
\[ \frac{\zeta(1-2u)}{-p^ry}-\frac{\zeta(1-2v)}{-p^ry} \equiv 0 \mod \Z_{(p)},\]
or equivalently 
\[\frac{B_{2u}}{2u} \equiv \frac{B_{2v}}{2v} \mod p^r\Z_{(p)}. \]

Part iv) implies that if $p-1 \nmid 2u$ and $p \nmid 2u-2$, then $\frac{B_{2u}}{2u} \in \Z_{(p)}$ which is also a special case of Kummer's congruence. See \cite[Theorem 3.2]{Arak} for more details on the classical Kummer's congruence. 

Part iii) recovers a special case of a congruence due to Carlitz \cite[Theorem 3]{Carl} which generalizes the $p$-local von Staudt-Clausen theorem. Indeed, we may assume $2u=zp^r(p-1)$, where $r \geq 0$ and $\gcd(z,p)=1$. Then Part iii) can be rewritten as 
\[\zeta(1-2u)+\frac{2-2u}{zp^{r+1} (zp^r+2)}\equiv 0 \mod \Z_{(p)},\]
since $p \nmid 2u-2$. By further simplifying we get
\[B_{2u} +\frac{1}{p} \equiv 1 \mod p^r\Z_{(p)}.\]
This is a congruence due to Carlitz \cite[Theorem 3]{Carl}.

\end{rem}

\section{An Example: $\Gamma_{\frac{(p-1)(p-2)}{2}}$ and a comparison of the chromatic and Brown-Quillen congruences}

In this section we apply Theorem \ref{theoremb} to the group $\Gamma_{\frac{(p-1)(p-2)}{2}}$, where $p \geq 5$. By a result of Broughton \cite[Section 4]{Broughton2} the group $\Gamma_{\frac{(p-1)(p-2)}{2}}$ has only one conjugacy class of subgroups isomorphic to $\Z/p \times \Z/p$ and does not have any higher rank elementary abelian $p$-subgroups. The genus $u=\frac{(p-1)(p-2)}{2}$ is the minimal genus for which a rank $2$ elementary abelian $p$-subgroup occurs.   Additionally, by \cite{Broughton2} the normalizer of $\Z/p \times \Z/p \leq \Gamma_{\frac{(p-1)(p-2)}{2}}$ is isomorphic to $\Sigma_3 \ltimes (\Z/p)^2$. Next, by \cite[Proposition 1.1]{Lev}, the group $GL_{(p-1)(p-2)} (\Z)$ does not contain $p^2$ torsion and hence nor does the symplectic group $Sp_{(p-1)(p-2)} (\Z)$. Since the Torelli group is torsion-free (see e.g., \cite[Theorem 6.8]{FarbMarg}), we conclude that $\Gamma_{\frac{(p-1)(p-2)}{2}}$ has no elements of order $p^2$. Now by Theorem \ref{theoremb} and Remark \ref{max card}, we have the chromatic congruence at height $n \geq 1$:
\[\chi_{\orb}(\Gamma_{\frac{(p-1)(p-2)}{2}})+(p^n-1)\sum_{\tiny{\begin{aligned}(H)\hspace{0.3cm}\\H \cong \Z/p\end{aligned}}} \chi_{\orb}(N(H))+(p^{2n}-(1+p)p^n+p)\chi_{\orb}(\Sigma_3 \ltimes (\Z/p)^2) \equiv 0 \mod \Z_{(p)}.\]
From here one obtains (see Lemma \ref{generating tuples formula})
\[\chi_{\orb}(\Gamma_{\frac{(p-1)(p-2)}{2}})+\frac{p^n-1}{p-1}\sum_{[g]}\chi_{\orb}(C\td{g})+(p^{2n}-(1+p)p^n+p)\cdot \frac{1}{6p^2} \equiv 0 \mod \Z_{(p)},\]
where $[g]$ runs over the conjugacy classes of order $p$ elements. Finally, for any $n \geq 1$, by \cite{HZ} and the proof of Theorem \ref{kummer} Part i), we see that the chromatic congruence formula at height $n\geq1$ looks as follows:
\[\frac{\zeta(1-2u)}{2-2u}+\frac{p^n-1}{p-1}\Big(-\frac{1}{p}\cdot \frac{\zeta(1-2v)}{2-2v} + \frac{(p-1)^2}{2u \cdot p(2u+2p-2)} \Big )+(p^{2n}-(1+p)p^n+p)\cdot \frac{1}{6p^2} \equiv 0 \mod \Z_{(p)},\]
where $u=\frac{(p-1)(p-2)}{2}$ and $v=\frac{p-1}{2}$. 
After taking the $p$-adic limit as $n \to \infty$, by the proof of Theorem \ref{Brown-Quillen}, we recover the Brown-Quillen congruence:
\[\frac{\zeta(1-2u)}{2-2u}-\frac{1}{p-1}\Big(-\frac{1}{p}\cdot \frac{\zeta(1-2v)}{2-2v} + \frac{(p-1)^2}{2u \cdot p(2u+2p-2)} \Big )+\frac{1}{6p} \equiv 0 \mod \Z_{(p)}.\]
On the other hand the chromatic congruence at height $n=1$ recovers the congruence of Theorem \ref{kummer} Part i)
\[\frac{\zeta(1-2u)}{2-2u}-\frac{1}{p} \cdot \frac{\zeta(1-2v)}{2-2v} +\frac{(p-1)^2}{2u \cdot p(2u+2p-2)}\equiv 0 \mod \Z_{(p)}.\]
The latter and Brown-Quillen congruence can be obtained from each other by using
\[\frac{\zeta(1-2v)}{2v-2} \equiv \frac{1}{p(p-1)(p-3)} \mod \Z_{(p)}\]
which is a consequence of the von Staudt-Clausen theorem. 


\bibliographystyle{amsalpha}
\bibliography{bib}

\vspace{1cm}

\begin{tabular}{l}
\scriptsize The Institute of Mathematics, University of Aberdeen, Fraser Noble building, AB24 3UE Aberdeen, Scotland, UK\\
\scriptsize  \textit{e-mail address:} \href{mailto:irakli.patchkoria@abdn.ac.uk}{irakli.patchkoria@abdn.ac.uk}
\end{tabular}

\end{document}